\documentclass[12pt]{amsart}
\usepackage{amsmath, amssymb, amsfonts}
\usepackage[all]{xypic}
\usepackage[pdftex]{graphicx}
\usepackage[usenames]{color}
\usepackage{stmaryrd}
\usepackage[utf8]{inputenc}

\theoremstyle{plain}
\newtheorem{theorem}{Theorem}[section]
\newtheorem{lemma}[theorem]{Lemma}
\newtheorem{corollary}[theorem]{Corollary}
\newtheorem{prop}[theorem]{Proposition}

\theoremstyle{remark}

\newtheorem{remark}[theorem]{Remark}

\newtheorem{example}[theorem]{Example}
\newtheorem{examples}[theorem]{Examples}
\newtheorem*{note*}{Note}
\newtheorem*{remark*}{Remark}
\newtheorem*{example*}{Example}

\theoremstyle{definition}
\newtheorem*{definition*}{Definition}
\newtheorem*{hypothesis*}{Hypothesis}
\newtheorem*{assumptions*}{Assumptions}
\newtheorem{definition}[theorem]{Definition}

\newcommand{\Z}{\mathbb{Z}}
\newcommand{\R}{\mathbb{R}}
\newcommand{\Q}{\mathbb{Q}}

\newcommand{\N}{\mathbb{N}}

\newcommand{\Gal}{\mathrm{Gal}}
\newcommand{\Tr}{\mathrm{Tr}}

\newcommand{\cl}{\mathrm{cl}}

\newcommand{\Hom}{\mathrm{Hom}}

\numberwithin{equation}{section}

\title[{On formal groups and Tate cohomology in local fields}]{On formal groups and\\ Tate cohomology in local fields}

\author{Nils Ellerbrock}
\address{Universit\"{a}t Bielefeld\\
	Fakult\"{a}t f\"{u}r Mathematik\\
	Postfach 100131\\
	Universit\"{a}tsstr. 25\\
	33501 Bielefeld\\
	Germany
}
\email{nils.ellerbrock@uni-bielefeld.de}
\author{Andreas Nickel}
\address{
Universit\"{a}t Bielefeld\\
Fakult\"{a}t f\"{u}r Mathematik\\
Postfach 100131\\
Universit\"{a}tsstr. 25\\
33501 Bielefeld\\
Germany}
\email{anickel3@math.uni-bielefeld.de}
\urladdr{http://www.math.uni-bielefeld.de/$\sim$anickel3/english.html}

\date{Version of 6th December 2016}
\begin{document}

\maketitle

\begin{abstract}
Let $L/K$ be a Galois extension of local fields of characteristic $0$
with Galois group $G$.
If $\mathcal{F}$ is a formal group over the ring of integers in $K$,
one can associate to $\mathcal F$ and 
each positive integer $n$ a $G$-module $F_L^n$ which as a set is
the $n$-th power of the maximal ideal of the ring of integers in $L$.
We give explicit necessary and sufficient conditions under which 
$F_L^n$ is a cohomologically trivial $G$-module.
This has applications to elliptic curves over local fields
and to ray class groups of number fields.
\end{abstract}

\section{Introduction}

Let $L/K$ be a Galois extension of local fields with Galois group $G$
and residue characteristic $p$.
We denote the ring of integers in $K$ and $L$ by
$\mathcal{O}_K$ and $\mathcal{O}_L$, respectively.
Let $\mathfrak P_L$ be the maximal ideal in $\mathcal{O}_L$
and let $n$ be a positive integer. Then $\mathfrak P_L^n$ is an
$\mathcal{O}_K[G]$-module in a natural way.
Köck \cite{MR2089083} has shown that $\mathfrak P_L^n$ is a projective
$\mathcal{O}_K[G]$-module if and only if $L/K$ is at most weakly ramified and
$n \equiv 1 \mod g_1$, where $g_1$ denotes the cardinality of the first
ramification group (which is the unique $p$-Sylow subgroup of
the inertia subgroup of $G$). As $\mathfrak P_L^n$ is torsionfree
as an $\mathcal{O}_K$-module,
it is $\mathcal{O}_K[G]$-projective if and only if 
it is a cohomologically trivial
$G$-module.

Now suppose that the local fields $L$ and $K$ are of characteristic $0$.
Then for sufficiently large $n$, the $p$-adic logarithm induces
$\Z_p[G]$-isomorphisms $U_L^n \simeq \mathfrak P_L^n$,
where $U_L^n := 1 + \mathfrak P_L^n$ are the principal units of level $n$.
It follows that the $\Z_p[G]$-module $U_L^n$ is cohomologically trival under
the same conditions on $L/K$ and $n$, at least if $n$ is sufficiently large.
Now it is very natural to ask whether this is still true for small $n$
and maybe as well for local fields of positive characteristic.
However, it is well known (and reproved in \S \ref{subsec:tame})
that $U_L^1$ is cohomologically trivial if and
only if $L/K$ is at most \emph{tamely} ramified.
Nevertheless, we give an affirmative answer to this question in \S 
\ref{sec:principal} whenever $n>1$. The proof is rather 
elementary and requires only some basic knowledge on Tate
cohomology and local class field theory.

Now let $\mathcal{F}$ be a formal group over $\mathcal{O}_K$
with formal group law $F$, where we again assume that the
characteristic of $K$ is $0$. Then for each positive integer $n$ one can
define a $\Z_p[G]$-module $F_L^n = \mathcal{F}(\mathfrak P_L^n)$ 
which as a set equals $\mathfrak P_L^n$,
but where addition is defined via $F$. Let $\mathbb{G}_a$ and $\mathbb{G}_m$
be the additive and the multiplicative formal group, respectively.
Then we have $\mathbb{G}_a(\mathfrak P_L^n) = \mathfrak P_L^n$
and $\mathbb{G}_m(\mathfrak P_L^n) \simeq U_L^n$.
This leads to the question whether $F_L^n$ is cohomologically trivial
under the same conditions on $L/K$ and $n$ as above. We again give an
affirmative answer whenever $n>1$ in \S \ref{sec:formal}.
Moreover, we show that $F_L^1$ is cohomologically trivial whenever
$L/K$ is tamely ramified.
Here, we build on results of Hazewinkel \cite{MR0349692}
on norm maps of formal groups.

Finally, we give two applications in \S \ref{sec:apps}.
First, we consider elliptic curves $E/K$. Then $E$ defines a formal
group over the ring of integers in $K$ such that we may apply
our main result of \S \ref{sec:formal}.
In particular, we deduce a generalization of the following classical
result of Mazur \cite{MR0444670}: When $E$ has good reduction and
$L/K$ is unramified,
then the norm map $E(L) \rightarrow E(K)$ is surjective.
In fact, our approach gives rise to similar results
when $E/K$ has additive or split multiplicative reduction.
Second, we consider finite Galois CM-extensions $L/K$
of number fields. Generalizing a result of the second author
\cite{MR2805422} we show that the minus $p$-part of certain
ray class groups is cohomologically trivial whenever
$L/K$ is weakly ramified above a fixed prime $p$. 
When $L/K$ is tamely ramified,
such a result was essential in the proof of the $p$-minus
part of the equivariant Tamagawa number conjecture for certain Tate motives
\cite{MR2805422, MR3552493}. We therefore believe that our result might be
useful in this direction as well.

\subsection*{Acknowledgements}
The authors acknowledge financial support provided by the DFG 
within the Collaborative Research Center 701
`Spectral Structures and Topological Methods in Mathematics'.

\subsection*{Notation and conventions}
All rings are assumed to have an identity element and all modules are assumed
to be left modules unless otherwise stated.

\section{Cohomology of principal units} \label{sec:principal}

\subsection{Tate cohomology}
Let $G$ be a finite group and let $M$ be a $\Z[G]$-module.
We denote by $M^G$ and $M_G$ the maximal submodule and maximal quotient
of $M$ upon which the action of $G$ is trivial, respectively.
For an integer $q$ we write $H^q(G,M)$ for the $q$-th Tate cohomology
group of $G$ with coefficients in $M$. We recall that for $q>0$
Tate cohomology coincides with usual group cohomology and that for
$q<-1$ we have $H^q(G,M) = H_{-q-1}(G,M)$, where the right hand side
denotes group homology of $G$ in degree $(-q-1)$.
Moreover, we have $H^0(G,M) = M^G / N_G(M)$, where 
$N_G := \sum_{\sigma \in G} \sigma \in \Z[G]$ and $N_G(M)$
denotes the image of the map 
\begin{eqnarray*}
N_G = N_{G,M}: M & \longrightarrow & M \\
m & \mapsto & N_G \cdot m. 
\end{eqnarray*}
Finally, we let $\Delta(G)$ be the kernel of the natural augmentation map
$\Z[G] \rightarrow \Z$ which sends each $\sigma \in G$ to $1$.
Then we have an equality $H^{-1}(G,M) = \ker(N_{G,M}) / \Delta(G) M$.

\begin{definition}
	Let $G$ be a finite group and let $M$ be a $\Z[G]$-module.
	Then $M$ is called \emph{cohomologically trivial} if
	$H^q(U,M) = 0$ for all $q \in \Z$ and all subgroups $U$ of $G$.
\end{definition}

\begin{remark}
	We note that $H^0(G,M)$ vanishes if and only if the norm map
	$N_{G,M}: M \rightarrow M^G$ is surjective.
\end{remark}

%\begin{remark} \label{rem:Sylow-ct}
%	Suppose that for every $p$-Sylow subgroup $P$ of $G$, there is an integer
%	$n_p$ such that $H^{n_p}(P,M) = H^{n_p+1}(P,M) = 0$.
%	Then $M$ is cohomologically trivial by \cite[Proposition 1.8.4]{MR2392026}.
%\end{remark}

We recall the following observation (see K\"ock 
\cite[Proof of Lemma 1.4]{MR2089083}).

\begin{lemma} \label{lem:spectral-argument}
	Let $N$ be a normal subgroup of $G$ and let $M$ be a $\Z[G]$-module.
	Suppose that $H^i(G/N,M^N) = 0$ and $H^i(N,M) = 0$ for all $i \in \Z$.
	Then $H^i(G,M) = 0$ for all $i \in \Z$.	
\end{lemma}

\begin{proof}
	This follows from the Hochschild--Serre spectral sequence
	\[
		H^p(G/N, H^q(N,M)) \implies H^{p+q}(G,M).
	\]
\end{proof}

Now suppose that $G$ is cyclic. Then for any $i \in \Z$
one has isomorphisms $H^i(G,M) \simeq H^{i+2}(G,M)$
by \cite[Proposition 1.7.1]{MR2392026}, and we let
\[
	h(M) := \frac{|H^0(G,M)|}{|H^1(G,M)|}
		= \frac{|H^{2r}(G,M)|}{|H^{2r+1}(G,M)|}, \quad r \in \Z
\]
be the \emph{Herbrand quotient} of $M$ (whenever the quotient on
the right hand side is well defined). The Herbrand quotient
is multiplicative on short exact sequences of $\Z[G]$-modules
(see \cite[Proposition 1.7.5]{MR2392026}).

\subsection{Principal units}
If $L$ is a local field, we denote the ring of integers in $L$
by $\mathcal{O}_L$. We note that $\mathcal{O}_L$ is a complete
discrete valuation ring and we let $v_L$ be the corresponding 
normalized valuation.
We put $\mathfrak P_L := \left\{x \in \mathcal{O}_L \mid v_L(x)> 0 \right\}$
which is the unique maximal ideal in $\mathcal{O}_L$.
The residue field $\lambda := \mathcal{O}_L / \mathfrak{P}_L$ is a finite
field of characteristic $p := \mathrm{char}(\lambda)>0$.
We let $U_L := \mathcal{O}_L^{\times}$ be the group of units in $L$.

\begin{definition}
	For each $n \in \N$ we put $U_L^n := 1 + \mathfrak{P}_L^n$
	and call $U_L^n$ the \emph{group of principal units of level $n$}.
\end{definition}

Each $U_L^n$ is a subgroup of $U_L$ of finite index. More precisely,
one has (non-canonical) isomorphisms
\begin{eqnarray}
U_L / U_L^1 & \simeq & \lambda^{\times} \label{eqn:first-nc} \\
U_L^n / U_L^{n+1} & \simeq & \lambda  \label{eqn:second-nc}
\end{eqnarray}
for all $n \in \N$.

\subsection{Ramification groups}
Let $L/K$ be a finite Galois extension of local fields with
Galois group $G$. We denote the residue field of $K$ by $\kappa$
and put $f := [\lambda: \kappa]$. Let $I$ be the inertia subgroup of $G$
and $e := |I|$ the ramification index. Then $G/I$ naturally identifies
with $\Gal(\lambda / \kappa)$ and we have $[L:K] = |G| = e \cdot f$.

\begin{definition}
	Let $i \geq -1$. Then we call
	\[
		G_i := \left\{\sigma \in G \mid v_L(\sigma(x) - x) \geq i+1
			\, \forall x \in \mathcal{O}_L \right\}
	\]
	the \emph{$i$-th ramification group} of the extension $L/K$.
	We let $g_i$ be the cardinality of $G_i$.
\end{definition}

We note that the ramification groups form a descending chain of 
normal subgroups of $G$ with abelian quotients 
(and thus the extension is solvable).
One has $G_{-1} = G$, $G_0 = I$ and $G_1$ is the (unique) $p$-Sylow
subgroup of $I$. We recall that the extension $L/K$ is said to be
\emph{unramified} if $G_0 = 1$, \emph{tamely ramified} if $G_1 = 1$ and
\emph{weakly ramified} if $G_2 = 1$.

If $H$ is a subgroup of $G$, we obviously have
$H_i  = G_i \cap H$. 
We define 
\begin{eqnarray*}
\phi = \phi_G: \left[-1, \infty \right) & \longrightarrow & 
	\left[ -1, \infty \right) \\
s & \mapsto & \int_{0}^{s} [G_0:G_t]^{-1} dt,
\end{eqnarray*}
where $[G_0:G_t] := [G_t:G_0]^{-1}$ if $t<0$. 
Then $G_i H / H = (G/H)_{\phi_H(i)}$ for every normal subgroup $H$ of $G$.
The map $\phi$
is piecewise linear and strictly increasing. We let
$\psi := \phi^{-1}$ be its inverse. For any $s \geq -1$ 
we then have the two inequalities
\begin{equation} \label{eqn:phi-psi-inequalities}
	\phi(s) \leq s, \quad s \leq \psi(s).
\end{equation}
We also recall from \cite[Chapter IV, \S 3]{MR554237} that
for $s \geq 0$ we have the formula
\begin{equation} \label{eqn:phi-formula}
\phi(s) = \frac{1}{g_0} \left((\sum_{i=1}^{\lfloor s \rfloor} g_i)
	+ (s - \lfloor s \rfloor) g_{\lceil s \rceil}\right).
\end{equation}
Here, we write $\lfloor s \rfloor$ for the largest integer  which
is less or equal to $s$, and $\lceil s \rceil$ for the least integer
which is greater or equal to $s$. 
Finally, we will frequently use the fact that $\psi(n)$
is an integer whenever $n \geq -1$ is an integer
\cite[Chapter IV, \S 3, Proposition 13]{MR554237}.

\subsection{Statement of the main result}
The main result of this section is the following theorem.

\begin{theorem} \label{thm:units-ct}
	Let $L/K$ be a finite Galois extension of local fields 
	with Galois group $G$. Let $n>1$
	be an integer. Then the $G$-module $U_L^n$ is
	cohomologically trivial if and only if 
	$L/K$ is at most weakly ramified
	and $n \equiv 1 \mod g_1$. Moreover, the $G$-module
	$U_L^1$ is cohomologically trivial if and only if 
	$L/K$ is at most tamely ramified.
\end{theorem}

The remaining part of this section is devoted to the proof
of Theorem \ref{thm:units-ct}.

\subsection{Galois invariants of principal units}
We first recall the following result on Galois invariants of ideals in $L$.

\begin{lemma} \label{lem:Koeck}
	Let $L/K$ be a finite Galois extension of local fields and let
	$n$ be an integer. Then we have an equality
	\[
		(\mathfrak{P}_L^n)^G = \mathfrak{P}_K^{1 + \lfloor \frac{n-1}{e} \rfloor}.
	\]
\end{lemma}

\begin{proof}
	This is \cite[Lemma 1.4 (a)]{MR2089083}.
\end{proof}

\begin{corollary} \label{cor:units-invariants}
	Let $L/K$ be a finite Galois extension of local fields and let
	$n \geq 1$ be an integer. Then we have an equality
	\[
	(U_L^n)^G = U_K^{1 + \lfloor \frac{n-1}{e} \rfloor}.
	\]
\end{corollary}

\begin{proof}
	As $U_L^n = 1 + \mathfrak{P}_L^n$, this is immediate
	from Lemma \ref{lem:Koeck}.
\end{proof}

\subsection{Cohomology in totally ramified extensions}
In this section we calculate the Herbrand quotient of $U_L$
when $L/K$ is a totally ramified cyclic extension. We begin with 
the following easy lemma.

\begin{lemma} \label{lem:valuation-zero}
	Let $L/K$ be a totally ramified Galois extension of local fields.
	Then the map $H^0(G, L^{\times}) \rightarrow H^0(G,\Z)$ 
	induced by the valuation $v_L: L^{\times} \rightarrow \Z$ 
	is trivial.
\end{lemma}

\begin{proof}
	We have $H^0(G,L^{\times}) = K^{\times} / N_G(L^{\times})$ and,
	as $L/K$ is totally ramified, we have $H^0(G, \Z) = \Z/ e \Z$.
	However, $v_L(x)$ is divisible by $e$ for every $x \in K^{\times}$. 
\end{proof}

\begin{remark}
	In fact, the map $H^i(G, L^{\times}) \rightarrow H^i(G,\Z)$
	induced by the valuation is trivial for every $i \in \Z$
	(see \cite[Chapter XII, \S 1, Exercise 2]{MR554237}).
	However, we will not need this more general statement.
\end{remark}

\begin{corollary} \label{cor:Herbrand}
	Let $L/K$ be a totally ramified cyclic Galois extension of local fields.
	Let $d := [L:K]$ be its degree. Then we have isomorphisms
	\[
		 H^i(G, U_L) \simeq \Z/ d\Z
	\]
	for every $i \in \Z$. In particular, we have $h(U_L) = 1$.
\end{corollary}

\begin{proof}
	We first observe that $H^{-1}(G, \Z) \simeq H^1(G, \Z) = \Hom(G, \Z) = 0$.
	As $H^1(G, L^{\times})$ vanishes by Hilbert's Theorem 90, 
	Lemma \ref{lem:valuation-zero} and
	the long exact sequence in (Tate) cohomology of the short exact sequence
	\[
		0 \longrightarrow U_L \longrightarrow L^{\times}
		\stackrel{v_L}{\longrightarrow} \Z \longrightarrow 0
	\]
	induces isomorphisms
	\[
		H^1(G, U_L) \simeq H^0(G, \Z) = \Z/d\Z
	\]
	and
	\[
		H^0(G, U_L) \simeq H^0(G,L^{\times}) \simeq H^{-2}(G, \Z) 
		\simeq G \simeq \Z / d \Z
	\]
	by local class field theory. As $H^i(G,U_L) \simeq H^{i+2}(G,U_L)$
	for every $i \in \Z$, we are done.
\end{proof}

\subsection{Tamely ramified extensions} \label{subsec:tame}
In this subsection we record two probably well-known results on the cohomology
of principal units in tamely ramified extensions. We give proofs for convenience.

\begin{prop} \label{prop:tamely-ramified-1}
	Let $L/K$ be a tamely ramified Galois extension of local fields.
	Then $U_L^n$ is cohomologically trivial for every integer $n \geq 1$.
\end{prop}

\begin{proof}
		As $U_L^n$ is a pro-$p$-group for every $n \geq 1$, we may and do
		assume that $G$ is a $p$-group. In particular, we may
		assume that $L/K$ is unramified. However, in this case
		the isomorphisms \eqref{eqn:first-nc} and \eqref{eqn:second-nc}
		are $G$-equivariant. 
		The cohomology of $\lambda$ and $\lambda^{\times}$ vanishes.
		Thus the result follows from the
		cohomological triviality of $U_L$ in unramified extensions
		\cite[Proposition 7.1.2]{MR2392026}.
\end{proof}

\begin{prop} \label{prop:tamely-ramified-2}
	The group of principal units $U_L^1$ is cohomologically trivial
	if and only if $L/K$ is tamely ramified.
\end{prop}

\begin{proof}
	If $L/K$ is tamely ramified, then $U_L^1$ is cohomologically trivial
	by Proposition \ref{prop:tamely-ramified-1}.
	Now suppose that $L/K$ is wildly ramified. Then there exists a subgroup
	of the inertia group of order $p$. Replacing $G$ by this subgroup
	we may assume that $G$ has order $p$ and that $L/K$ is totally
	ramified. As the index of $U_L^1$ in $U_L$ is finite of order
	prime to $p$ by \eqref{eqn:first-nc}, we then have
	isomorphisms $H^i(G, U_L^1) \simeq H^i(G, U_L)$ for all $i \in \Z$.
	Now Corollary \ref{cor:Herbrand} implies that
	$U_L^1$ is not cohomologically trivial.
\end{proof}

\subsection{Weakly ramified extensions}
Our first task in this subsection is to prove an analogue
of Proposition \ref{prop:tamely-ramified-1} for weakly
ramified extensions.

\begin{prop} \label{prop:weakly-implies-ct}
	Let $L/K$ be a weakly ramified Galois extension of local fields.
	Let $n>1$ be an integer such that $n \equiv 1 \mod g_1$.
	Then $U_L^n$ is cohomologically trivial.
\end{prop}

\begin{proof}
	Let $H$ be a subgroup of $G$. Then $L / L^H$ is also weakly ramified
	and $n \equiv 1 \mod |H_1|$.
	So we may and do assume that $H=G$.
	As $G$ is solvable, Lemma \ref{lem:spectral-argument} and Proposition
	\ref{prop:tamely-ramified-1} imply that we may further assume
	that $G$ is cyclic of order $p$, and that $L/K$ is totally ramified.
	Then for $s \geq 1$ we have
	\[
		\phi(s) = \frac{1}{p}(p+s-1)
	\]
	by equation \eqref{eqn:phi-formula}.
	This gives the second equality in the computation
	\begin{eqnarray*}
		N_G(U_L^n) & = & U_L^{\lceil \phi(n) \rceil}\\
		& = & U_L^{1 + \lceil \frac{n-1}{p} \rceil}\\
		& = & U_L^{1 + \lfloor \frac{n-1}{p} \rfloor}\\
		& = & (U_L^n)^G,
	\end{eqnarray*}
	whereas the first is \cite[Chapter V, \S 3, Corollary 4]{MR554237},
	the third holds as $n \equiv 1 \mod p$ by assumption,
	and the last equality is Corollary \ref{cor:units-invariants}.
	It follows that $H^0(G, U_L^n)$ vanishes.
	As the Herbrand quotient of a finite module is trivial
	\cite[Chapter VIII, \S 4, Proposition 8]{MR554237},
	it follows from \eqref{eqn:first-nc}, \eqref{eqn:second-nc}
	and Corollary \ref{cor:Herbrand} that
	\[
		h(U_L^n) = h(U_L) = 1.
	\]
	Thus $U_L^n$ is cohomologically trivial 
	%by Remark \ref{rem:Sylow-ct} 
	as desired.
\end{proof}

We now prove the following converse 
of Proposition \ref{prop:weakly-implies-ct}.

\begin{prop} \label{prop:ct-implies-weak}
	Let $L/K$ be a finite Galois extension of local fields 
	with Galois group $G$ and
	let $n \geq 1$ be an integer. Suppose that $U_L^n$ is 
	cohomologically trivial as $G$-module. Then it holds:
	\begin{enumerate}
		\item 
		We have that $n \equiv 1 \mod g_1$.
		\item
		The extension $L/K$ is at most weakly ramified.
	\end{enumerate}
\end{prop}

\begin{proof}
	By Proposition \ref{prop:tamely-ramified-2} we may assume
	that $n>1$. If $U_L^n$ is cohomologically trivial as a
	$G$-module, then in particular as a $G_1$-module.
	We may therefore assume that $G = G_1$ and also that $G$ is
	non-trivial.
	Then there is an integer  $k \geq 1$ such that
	$|G| = g_1 =  p^k$. We put $m :=\lfloor \phi(n-1) \rfloor$. Then
	by \eqref{eqn:phi-formula} we have
	\begin{equation} \label{eqn:m-inequality}
	m = \lfloor \phi(n-1) \rfloor
	= \left\lfloor \frac{1}{p^k}\sum_{i = 1}^{n-1} g_i \right\rfloor
	= 1 + \left\lfloor \frac{1}{p^k}\sum_{i = 2}^{n-1} g_i \right\rfloor
		\geq 1 + \left\lfloor \frac{n-2}{p^k} \right\rfloor
		\geq \left\lfloor \frac{n-1}{p^k} \right\rfloor.
	\end{equation}
	We now consider the following chain of inclusions:
	\begin{equation} \label{eqn:units-inclusions}
		N_G(U_L^n) \subseteq N_G(U_L^{\psi(m)+1})
		\subseteq U_K^{m+1}
		\subseteq U_K^{1 + \lfloor \frac{n-1}{p^k} \rfloor}.
	\end{equation}
	Here, the first inclusion follows from $m \leq \phi(n-1)$
	and thus $\psi(m) + 1 \leq n$. The second inclusion is
	\cite[Chapter V, \S 6, Proposition 8]{MR554237} and the last is due to
	\eqref{eqn:m-inequality}. However, as $H^0(G, U_L^n)$ vanishes,
	Corollary \ref{cor:units-invariants} implies that
	\[
		U_K^{1 + \lfloor \frac{n-1}{p^k} \rfloor} = (U_L^n)^G = N_G(U_L^n).
	\]
	Thus all inclusions in \eqref{eqn:units-inclusions} are in fact 
	equalities and therefore
	$m = \lfloor \frac{n-1}{p^k} \rfloor$. It now follows from
	\eqref{eqn:m-inequality} that
	\[
		 1 + \left\lfloor \frac{n-2}{p^k} \right\rfloor
		= \left\lfloor \frac{n-1}{p^k} \right\rfloor
	\]
	and thus (i) holds. Now choose $t \in \Z$ such that $G_t \not=1$
	and $G_{t+1} = 1$. Note that $t \geq 1$.
	If $H$ is any subgroup of $G$, then one has $H_i = H \cap G_i$
	for all $i \geq -1$. So we may further assume that
	$k=1$ and thus we have
	\[
		G = G_{-1} = G_0 = \dots = G_t \simeq \Z/ p \Z.
	\] 
	In this special situation we have
	\[
		\phi(s) = \left\{ \begin{array}{lll}
		s & \mbox{ if } & s \leq t, \\
		t + \frac{s-t}{p} & \mbox{ if } & s \geq t.
		\end{array} \right.
	\]
	As $m = \lfloor \phi(n-1) \rfloor = \lfloor \frac{n-1}{p} \rfloor$,
	we have $n-1>t$ and thus
	\[
		\left\lfloor \frac{n-1}{p} \right\rfloor 
		= t + \left\lfloor \frac{n - 1 -t}{p} \right\rfloor.
	\]
	This is only possible if $t=1$ and therefore $G_2$ vanishes.
\end{proof}
 
\begin{proof}[Proof of Theorem \ref{thm:units-ct}]
	This now follows easily from Propositions \ref{prop:tamely-ramified-2},
	\ref{prop:weakly-implies-ct} and \ref{prop:ct-implies-weak}.
\end{proof}

\section{Formal groups} \label{sec:formal}

The aim of this section is to generalize Theorem \ref{thm:units-ct}
to arbitrary formal groups over local fields of characteristic $0$.

\subsection{Basic definitions and examples}

\begin{definition}
	Let $R$ be a commutative ring.
	A (commutative) \emph{formal group} $\mathcal{F}$ over $R$ is given by a power series
	$F(X,Y) \in R\llbracket X,Y \rrbracket$ with the following properties:
	\begin{enumerate}
		\item
		$F(X,Y) \equiv X + Y \mod (\deg 2)$.
		\item
		$F(X,F(Y,Z)) = F(F(X,Y),Z)$.
		\item 
		$F(X,Y) = F(Y,X)$.
	\end{enumerate}
	The power series $F$ is called the \emph{formal group law}
	of the formal group $\mathcal{F}$.
\end{definition}

If $R$ is a complete discrete valuation ring, then one can associate
proper groups to a formal group over $R$.

\begin{definition}
	Let $R$ be a complete discrete valuation ring with maximal
	ideal $\mathfrak{m}$. Let $\mathcal F$ be a formal group over $R$
	with formal group law $F \in R\llbracket X,Y \rrbracket$.
	Then for every positive integer $n$ one can associate to $\mathcal{F}$
	the group $\mathcal{F}(\mathfrak m^n)$ which as a set equals
	$\mathfrak m^n$ with the new group law
	\[
		x +_{\mathcal{F}} y = F(x,y), \quad 
		x,y \in \mathcal{F}(\mathfrak m^n).
	\]
	We will call $\mathcal{F}(\mathfrak m^n)$ the 
	\emph{associated group of level $n$}.
\end{definition}

\begin{examples}
	Let $R$ be a commutative ring.
	\begin{enumerate}
		\item 
		The power series $F(X,Y) = X+Y$ obviously defines a formal group
		which is called the \emph{additive formal group} and will be
		denoted by $\mathbb G_a$. If $R$ is a complete discrete
		valuation ring with maximal ideal $\mathfrak m$, then one
		has $\mathbb{G}_a(\mathfrak m^n) = \mathfrak m^n$
		as groups for every positive integer $n$.
		\item
		The power series $F(X,Y) = X + Y + XY$ defines a formal group
		which is called the \emph{multiplicative formal group} and will be
		denoted by $\mathbb G_m$. If $R = \mathcal{O}_L$
		is the ring of integers in a local field $L$, then one has
		canonical isomorphisms
		$\mathbb{G}_m(\mathfrak P_L^n) \simeq U_L^n$ for every $n \in \N$.
		\item
		Let $L$ be a local field of characteristic $0$ and let
		$\pi \in \mathfrak P_L$ be a uniformizer.
		Let $q$ denote the cardinality of the residue field $\lambda$.
		Then for every power series 
		$f \in \mathcal{O}_L \llbracket Z \rrbracket$ such that
		$f(Z) \equiv Z^q \mod \pi$ and $f(Z) \equiv \pi Z \mod Z^2$
		there is a unique power series $F(X,Y) \in
		\mathcal{O}_L \llbracket X,Y \rrbracket$ such that
		$F(X,Y) \equiv X + Y \mod (\deg 2)$ and
		$f(F(X,Y)) = F(f(X),f(Y))$. This power series defines a formal
		group  $\mathcal{F}_{\pi}$ over $\mathcal O_L$,
		the \emph{Lubin--Tate formal group} associated to $\pi$.
		In fact, $\pi$ determines the formal group $\mathcal{F}_{\pi}$
		up to isomorphism. We refer the reader to
		\cite{MR0172878} for more details.
		\item
		Let $L$ be a local field of characteristic $0$
		and let $E/L$ be an elliptic curve given by a
		minimal Weierstraß equation. Then $E$ defines a
		formal group $\hat E$ over $\mathcal{O}_L$. The associated group
		of level $n$ will be denoted by $E_L^n$. Then 
		one has isomorphisms
		\[
			E_L^n \simeq E_n(L) := \left\{(x,y) \in E(L) \mid
			v_L(x) \leq -2n \right\} \cup \left\{O\right\}.
		\]
		Indeed by \cite[Chapter VII, Proposition 2.2]{MR2514094}
		the map $E_1(L) \rightarrow E_L^1$, $(x,y) \mapsto - \frac{x}{y}$
		is an isomorphism. As $2 v_L(y) = 3 v_L(x)$, we have
		$- \frac{x}{y} \in E_L^n$ if and only if $v_L(x) \leq -2n$
		(see also \cite[Exercise 7.4]{MR2514094}).
	\end{enumerate}
\end{examples}

\subsection{Galois invariants}
For the rest of this section we let $L/K$ be a finite Galois extension
of local fields of characteristic $0$. 
Recall the notation of \S \ref{sec:principal}. In particular,
we have $G = \Gal(L/K)$ and $e$ denotes the ramification index.
If $\mathcal F$ is a formal group over
$\mathcal{O}_K$ with formal group law 
$F \in \mathcal{O}_K \llbracket X,Y \rrbracket$, we put
$F_L^n := \mathcal{F}(\mathfrak{P}_L^n)$ and
$F_K^n := \mathcal{F}(\mathfrak{P}_K^n)$
for every positive integer $n$.
As $\mathcal{F}$ is defined over $\mathcal{O}_K$, the Galois
group $G$ acts on the associated groups $F_L^n$.

\begin{lemma} \label{lem:F_invariants}
	Let $\mathcal F$ be a formal group over
	$\mathcal{O}_K$ and let $n>0$ be an integer. Then
	we have an equality
	\[
		(F_L^n)^G = F_K^{1 + \lfloor \frac{n-1}{e} \rfloor}.
	\]
\end{lemma}

\begin{proof}
	This is immediate from Lemma \ref{lem:Koeck}.
\end{proof}

\subsection{The norm map}

Let $\mathcal F$ be a formal group over
$\mathcal{O}_K$ and let $k>0$ be an integer.
If $x_1, \dots, x_k$ belong to $F_L^1$, we let
\[
	\sum_{j \in J}^{\mathcal{F}} x_j
	:= x_1 +_{\mathcal{F}} \dots +_{\mathcal{F}} x_k,
\]
where $J = \left\{1,\dots,k \right\}$
and similarly for other index sets $J$.

\begin{definition}
	Let $\mathcal F$ be a formal group over
	$\mathcal{O}_K$. We define a \emph{norm map}
	\begin{eqnarray*}
	N_G^{\mathcal{F}}: F_L^1 & \longrightarrow & F_K^1\\
	x & \mapsto & \sum_{\sigma \in G}^{\mathcal{F}} \sigma(x).
	\end{eqnarray*}
	We let $\Tr_{L/K} := N_G^{\mathbb G_a}$ and
	$N_{L/K} := N_G^{\mathbb G_m}$ be the usual trace and norm maps,
	respectively.
\end{definition}

\begin{lemma} \label{lem:Hazewinkel}
	Let $\mathcal F$ be a formal group over
	$\mathcal{O}_K$ and let $x \in F_L^1$.
	Then there are $a_i \in \mathcal{O}_K$, $1 \leq i < \infty$,
	such that
	\[
		N_G^{\mathcal{F}}(x) \equiv \Tr_{L/K}(x) + 
		\sum_{i=1}^{\infty} a_i (N_{L/K}(x))^i \mod \Tr_{L/K}(x^2 \mathcal{O}_L).
	\]
\end{lemma}

\begin{proof}
	This is \cite[Corollary 2.4.2]{MR0349692}.
\end{proof}

\subsection{Norm maps in totally ramified extensions}

We first prove a generalization of
\cite[Chapter V, \S 3, Proposition 4]{MR554237}.
Let $t \in \Z$ be the last ramification jump, that is $G_t \not= 1$
and $G_{t+1} = 1$.

\begin{prop} \label{prop:F-inclusions}
	Let $\ell$ be a prime and suppose that $L/K$ is totally ramified
	and cyclic of degree $\ell$. Then for every $n \in \N$ we have
	inclusions
	\begin{enumerate}
		\item 
		$N_G^{\mathcal{F}}(F_L^{\psi(n)}) \subseteq F_K^n$ and
		\item
		$N_G^{\mathcal{F}}(F_L^{\psi(n)+1}) \subseteq F_K^{n+1}$.
	\end{enumerate}
\end{prop}

\begin{proof}
	We only prove (i), the proof of (ii) being similar.
	Let $x \in \mathfrak P_L^{\psi(n)}$.
	By Lemma \ref{lem:Hazewinkel}
	it suffices to show that $v_K(\Tr_{L/K}(x)) \geq n$
	and $v_K(N_{L/K}(x)) \geq n$. As $L/K$ is totally ramified,
	we indeed have
	\[
		v_K(N_{L/K}(x)) = v_L(x) \geq \psi(n) \geq n,
	\]
	where the last inequality is \eqref{eqn:phi-psi-inequalities}.
	For the trace one knows by \cite[Chapter V, \S 3, Lemma 4]{MR554237} that
	\begin{equation} \label{eqn:trace-inequality}
		v_K(\Tr_{L/K}(x)) \geq
		\left\lfloor \frac{(t+1)(\ell-1) + \psi(n)}{\ell} \right\rfloor.
	\end{equation}
	If $n \leq t$ then $\psi(n) = n$ and so \eqref{eqn:trace-inequality}
	implies that
	\[
		v_K(\Tr_{L/K}(x)) \geq
		\left\lfloor \frac{(n+1)(\ell -1) + n}{\ell} \right\rfloor = 
		\left\lfloor \frac{n \ell + \ell - 1}{\ell} \right\rfloor = n.
	\]
	If on the other hand $n \geq t$ we have $\psi(n) = t + \ell (n-t)$.
	Thus \eqref{eqn:trace-inequality} simplifies to
	\[
	v_K(\Tr_{L/K}(x)) \geq
	\left\lfloor \frac{n \ell + \ell - 1}{\ell} \right\rfloor = n
	\]
	as desired.
\end{proof}

\begin{corollary} \label{cor:norm-F-equality}
	Let $\ell$ be a prime and suppose that $L/K$ is totally ramified 
	and cyclic of degree $\ell$. Then for every $n>t$ we have
	\[
		N_G^{\mathcal{F}}(F_L^{\psi(n)}) = F_K^n \quad \mbox{and} \quad
		N_G^{\mathcal{F}}(F_L^{\psi(n)+1}) =  F_K^{n+1}.	
	\]
\end{corollary}

\begin{proof}
	Proposition \ref{prop:F-inclusions} implies that the norm map
	$N_G^{\mathcal{F}}$ induces maps
	\[
	N_n^{\mathcal{F}}: F_L^{\psi(n)} / F_L^{\psi(n)+1}
	\longrightarrow F_K^n / F_K^{n+1}
	\]
	for every $n \in \N$.
	Now let $m \geq n > t$ be integers. Then by 
	\cite[Chapter IV, Proposition 3.2]{MR2514094} 
	we have a commutative diagram
	\[ \xymatrix{
		F_L^{\psi(m)} / F_L^{\psi(m)+1} \ar[rr]^{N_m^{\mathcal{F}}} 
			\ar[d]^{\simeq} & & F_K^m / F_K^{m+1} \ar[d]^{\simeq} \\
		U_L^{\psi(m)} / U_L^{\psi(m)+1} \ar[rr]^{N_m^{\mathbb{G}_m}} 
			& & U_K^m / U_K^{m+1}
	}\]
	The maps $N_m^{\mathbb{G}_m}$ are surjective 
	by \cite[Chapter V, \S 3, Corollary 2]{MR554237}.
	Thus the maps $N_m^{\mathcal{F}}$ are also surjective and likewise
	\[
		F_L^{\psi(m)} / F_L^{\psi(m+1)} \twoheadrightarrow
		F_L^{\psi(m)} / F_L^{\psi(m)+1} \twoheadrightarrow
		F_K^m / F_K^{m+1}
	\]
	for every $m \geq n > t$. Now \cite[Chapter V, \S 1, Lemma 2]{MR554237}
	implies that $N_G^{\mathcal{F}}: F_L^{\psi(n)} \rightarrow F_K^n$
	is surjective. The second equality follows from the first
	and Proposition \ref{prop:F-inclusions} via the chain of inclusions
	\[
		F_K^{n+1} = N_G^{\mathcal{F}}(F_L^{\psi(n+1)})
		\subseteq N_G^{\mathcal{F}}(F_L^{\psi(n)+1}) \subseteq F_K^{n+1}.
	\]
\end{proof}

\begin{corollary} \label{cor:better-norm-F}
	Let $\ell$ be a prime and suppose that $L/K$ is totally ramified 
	and cyclic of degree $\ell$. Then for every $v>t$, $v \in \R$ we have
	\[
	N_G^{\mathcal{F}}(F_L^{\lceil\psi(v)\rceil}) = F_K^{\lceil v \rceil}.	
	\]
\end{corollary}

\begin{proof}
	The proof is completely 
	analogous to \cite[Chapter V, \S 3, Corollary 4]{MR554237}
	using Proposition \ref{prop:F-inclusions} and Corollary
	\ref{cor:norm-F-equality}.
\end{proof}

\begin{corollary} \label{cor:F-inclusions}
	Suppose that $L/K$ is totally ramified. 
	Then for every $n \in \N$ we have
	inclusions
	\[
		N_G^{\mathcal{F}}(F_L^{\psi(n)}) \subseteq F_K^n 
		\quad \mbox{and} \quad
		N_G^{\mathcal{F}}(F_L^{\psi(n)+1}) \subseteq F_K^{n+1}.
	\]
\end{corollary}

\begin{proof}
	As the Galois extension $L/K$ is solvable, this follows from
	Proposition \ref{prop:F-inclusions} by induction.
\end{proof}

\subsection{The cohomology of the associated groups}

We start with the following auxiliary result.

\begin{lemma} \label{lem:Herbrand}
	Suppose that $L/K$ is cyclic of prime degree.
	Then $h(F_L^n) = 1$ for every $n \in \N$.
\end{lemma}

\begin{proof}
		As the index of $F_L^{n+1}$ in $F_L^n$ is finite, the Herbrand
		quotient $h(F_L^n)$ does not depend on $n$. So we may assume
		that $n$ is sufficiently large such that the formal
		logarithms of $\mathcal{F}$ and $\mathbb G_m$ induce
		an isomorphism $F_L^n \simeq U_L^n$. If $L/K$ is
		tamely ramified, we have $h(U_L^n) = 1$ by Theorem
		\ref{thm:units-ct}. If $L/K$ is wildly ramified, then it is
		totally ramified and thus $h(U_L^n) = h(U_L) = 1$
		by Corollary \ref{cor:Herbrand}.
\end{proof}

\begin{prop} \label{prop:F-tamely-ramified}
	Let $L/K$ be tamely ramified. Then the associated groups
	$F_L^n$ are cohomologically trivial for every $n \in \N$.
\end{prop}

\begin{proof}
	If $L/K$ is tamely ramified, then the ideals
	$\mathfrak P_L^n$ are cohomologically trivial by
	\cite[Theorem 1.1]{MR2089083}. As we have isomorphisms
	$F_L^n / F_L^{n+1} \simeq \mathfrak P_L^n / \mathfrak P_L^{n+1}$
	for every $n \in \N$, it suffices to show that $F_L^1$ is
	cohomologically trivial. 
	By Lemma \ref{lem:spectral-argument} we may suppose that $L/K$
	is cyclic of prime degree.
	The norm map 
	$N_G^{\mathcal{F}}: F_L^1 \rightarrow F_K^1$ is surjective
	by \cite[Proposition 3.1]{MR0349692} and thus
	$H^0(G, F_L^1)$ vanishes. Now the result follows from
	Lemma \ref{lem:Herbrand}.
\end{proof}

\begin{prop} \label{prop:F-weakly-I}
	Let $L/K$ be weakly ramified and let $n>1$ be an integer
	such that $n \equiv 1 \mod g_1$.
	Then $F_L^n$ is cohomologically trivial.
\end{prop}

\begin{proof}
	By Proposition \ref{prop:F-tamely-ramified}
	and Lemma \ref{lem:spectral-argument} we may and do
	assume that $L/K$ is cyclic of order $p$ and totally ramified. 
	We then have
	\[
		(F_L^n)^G = F_K^{1 + \lfloor \frac{n-1}{p} \rfloor} =
		F_K^{1 + \lceil \frac{n-1}{p} \rceil}
		= N_G^{\mathcal{F}}(F_L^n).
	\]	 
	Here, the first and last equality follow from
	Lemma \ref{lem:F_invariants} and  
	Corollary \ref{cor:better-norm-F}, respectively. 
	As $n \equiv 1 \mod p$, the remaining equality is also clear.
	We obtain $H^0(G,F_L^n) = 1$, and
	Lemma \ref{lem:Herbrand} again implies the result.
\end{proof}

We are now in a position to state and prove the main result of this section.

\begin{theorem} \label{thm:main-result}
	Let $L/K$ be a finite Galois extension of local fields of characteristic
	$0$ with Galois group $G$. Let $\mathcal{F}$ be a formal group
	over $\mathcal{O}_K$ with formal group law $F$ and let $n>1$ 
	be an integer. Then $F_L^n$ is a cohomologically trivial
	$G$-module if and only if $L/K$ is weakly ramified and
	$n \equiv 1 \mod g_1$.
\end{theorem}

\begin{proof}
	Suppose that $n>1$ is an integer such that $F_L^n$ is
	cohomologically trivial. Then the same reasoning as in the proof 
	of Proposition \ref{prop:ct-implies-weak} using
	Lemma \ref{lem:F_invariants} and Corollary
	\ref{cor:F-inclusions} shows that $L/K$ is weakly
	ramified and $n \equiv 1 \mod g_1$.
	The converse also holds by Proposition 
	\ref{prop:F-weakly-I}.
\end{proof}

\begin{remark}
	It is in general \emph{not} true that $F_L^1$ is cohomologically
	trivial if and only if $L/K$ is tamely ramified.
	In fact, even in the case $\mathcal{F} = \mathbb{G}_a$ Köck's
	result \cite[Theorem 1.1]{MR2089083} shows that $\mathbb{G}_a(\mathfrak P_L) = \mathfrak P_L$
	is cohomologically trivial if and only if $L/K$ is weakly
	ramified. We now give a second example.
\end{remark}

\begin{example}
	We denote the absolute Galois group of $\mathbb Q_p$ by $G_{\mathbb Q_p}$
	and let $\chi:G_{\mathbb Q_p} \rightarrow \Z_p^{\times}$ be
	an unramified character. Let $K / \Q_p$ be an unramified extension
	and put $\pi := p \chi(\varphi)$, where $\varphi$ denotes the absolute 
	Frobenius of $\Q_p$. We consider the Lubin--Tate formal group
	$\mathcal{F} = \mathcal{F}_{\pi}$ over $\mathcal{O}_K$.
	Now let $L/K$ be a finite Galois extension with Galois group $G$
	and suppose that the absolute Frobenius $\varphi_L$
	of $L$ does not belong to the kernel of $\chi$.
	We put $\omega_L := v_{\Q_p}(1 - \chi(\varphi_L))$ and observe
	that $\omega_L \geq 0$. Then by 
	\cite[Lemma 4.1.1 and Theorem 4.2.3]{epsilon} there is 
	a $\Z_p[G]$-module $I$ and an
	exact sequence of $\Z_p[G]$-modules
	\[
		0 \longrightarrow F_L^1 \longrightarrow I
		\longrightarrow I \longrightarrow \Z / p^{\omega_L} \Z (\chi)
		\longrightarrow 0.
	\]
	The $\Z_p[G]$-module $I$ is cohomologically trivial by
	\cite[Lemma 4.1.2]{epsilon}. 
	Now suppose that $L/K$ is weakly and wildly ramified.
	Then $G_1 \not= 1$ and we have isomorphisms
	\[
		H^i(G_1, F_L^1) \simeq H^{i-2}(G_1,\Z / p^{\omega_L} \Z)
	\]
	for every $i \in \Z$. It follows that $F_L^1$ is cohomologically
	trivial if and only if $\omega_L = 0$.
\end{example}

\section{Applications} \label{sec:apps}

\subsection{Elliptic curves}

Let $L/K$ be a finite Galois extension of local fields of characteristic
$0$ and let $E/K$ be an elliptic curve given by a minimal Weierstraß equation
\[
	E: y^2 + a_1 xy + a_3 y = x^3 + a_2 x^2 + a_4 x + a_6
\]
with discriminant $\Delta \in \mathcal{O}_K$.
We denote the reduction of $E$ by $\overline{E}$ which is a not
necessarily smooth curve over the residue field $\kappa$.
We denote by $\overline E_{ns}(\lambda)$ the subset of
$\overline E(\lambda)$ comprising all non-singular points.
We let $E_0(L)$ be the set of  $L$-rational points of $E$
which have non-singular reduction. 
By \cite[Chapter VII, Proposition 2.1]{MR2514094}
one then has a short exact sequence
\begin{equation} \label{eqn:elliptic-curve-ses}
	0 \longrightarrow E_1(L) \longrightarrow E_0(L) \longrightarrow
	\overline{E}_{ns}(\lambda) \longrightarrow 0.
\end{equation}

\begin{prop} \label{prop:E-unramified}
	Suppose that $L/K$ is unramified and that $E/K$ has good reduction.
	Then the set $E(L)$ of $L$-rational points is a cohomologically trivial
	$G$-module.
\end{prop}

\begin{proof}
	We observe
	that we have $E_0(L) = E(L)$ and 
	$\overline{E}_{ns}(\lambda) = \overline{E}(\lambda).$
	Now $E_L^1$ is cohomologically trivial by Proposition
	\ref{prop:F-tamely-ramified}, whereas $\overline{E}(\lambda)$
	is cohomologically trivial by \cite[Proposition 3]{MR0086367}.
	Moreover, we have $E_L^1 \simeq E_1(L)$ by
	\cite[Chapter VII, Proposition 2.2]{MR2514094}.
	Now the result follows from sequence \eqref{eqn:elliptic-curve-ses}.
\end{proof}

\begin{remark}
	If $L/K$ is unramified and $E$ has good reduction, then
	Proposition \ref{prop:E-unramified} in particular implies
	that the norm map $E(L) \rightarrow E(K)$ is surjective.
	This is a classical result of Mazur \cite[Corollary 4.4]{MR0444670}.
\end{remark}

\begin{prop}
	Suppose that $L/K$ is unramified and that $E/K$ has additive reduction. Then the set $E_0(L)$ is a cohomologically trivial
	$G$-module. If in addition the order of $G$ is prime to $6$,
	then also $E(L)$ is a cohomologically trivial $G$-module.
\end{prop}

\begin{proof}
	 We know that $E_L^1 \simeq E_1(L)$ is cohomologically trivial by Proposition
	 \ref{prop:F-tamely-ramified}. 
	 Likewise $\overline{E}_{ns}(\lambda) \simeq \lambda$
	 is cohomologically trivial, as $L/K$ is unramified. It follows
	 from sequence \eqref{eqn:elliptic-curve-ses} that
	 $E_0(L)$ is cohomologically trivial. As the index of $E_0(L)$ in
	 $E(L)$ is at most $4$ by the theorem of Kodaira and Néron
	 \cite[Chapter VII, Theorem 6.1]{MR2514094}, the final claim is also clear.
\end{proof}

\begin{prop}
	Suppose that $L/K$ is unramified and that $E/K$ has split
	multiplicative reduction.
	Then the set $E_0(L)$ is a cohomologically trivial
	$G$-module. If in addition the order of $G$ is prime to
	$v_K(\Delta)$, then also $E(L)$ is a cohomologically trivial $G$-module.
\end{prop}

\begin{proof}
	The first claim again follows from Proposition \ref{prop:F-tamely-ramified}
	and sequence \eqref{eqn:elliptic-curve-ses} once we observe
	that $\overline{E}_{ns}(\lambda) \simeq \lambda^{\times}$
	is cohomologically trivial. Likewise, the final claim again follows
	from the theorem of Kodaira and Néron
	\cite[Chapter VII, Theorem 6.1]{MR2514094} which says that
	in this case $E(L) / E_0(L)$ is a cyclic group of order
	$v_L(\Delta) = v_K(\Delta)$.
\end{proof}

\subsection{Ray class groups}

In this section we let $L/K$ be a finite Galois extension of number fields
with Galois group $G$. If $\mathfrak{m}$ is an ideal of the ring
of integers $\mathcal{O}_L$ in $L$, we write $\cl_L^{\mathfrak m}$
for the ray class group of $L$ to the ray $\mathfrak m$.
We say that the modulus $\mathfrak m$ is $G$-equivariant
if $\sigma(\mathfrak m) = \mathfrak m$ for every $\sigma \in G$.
In this case the ray class group $\cl_L^{\mathfrak m}$ is endowed with
a natural $G$-action.

Now suppose that $L/K$ is a CM-extension, so $K$ is totally real,
$L$ is totally complex and complex conjugation induces a unique
automorphism $j \in G$ which is indeed central in $G$. If $M$
is a $G$-module, we denote by $M^+$ and $M^-$ the submodules of $M$
upon which $j$ acts by $+1$ and $-1$, respectively.
When $\mathfrak m$ is $G$-equivariant, we put 
$A_L^{\mathfrak m} := (\cl_L^{\mathfrak m})^-$.

Now let $p$ be a prime. For every prime $\mathfrak p$ in $K$ above $p$
we choose a prime $\mathfrak P$ in $L$ above $\mathfrak p$.
We say that $L/K$ is weakly ramified above $p$ if the
local extensions $L_{\mathfrak P} / K_{\mathfrak p}$ are weakly ramified
for all primes $\mathfrak p$ in $K$ above $p$. We let $e_{\mathfrak p}$ be the
ramification index of the extension $L_{\mathfrak P} / K_{\mathfrak p}$
which does not depend on the choice of the prime $\mathfrak P$.
We write $e_{\mathfrak p} = p^{k_{\mathfrak p}} e_{\mathfrak p}'$, where $k_{\mathfrak p}$ and $e_{\mathfrak p}'$ are integers
such that $e_{\mathfrak p}'$ is not divisible by $p$.

\begin{theorem} \label{thm:ray-class-ct}
	Let $L/K$ be a Galois CM-extension of number fields with Galois group $G$.
	Let $p$ be an odd prime and suppose that $L/K$ is weakly ramified
	above $p$. Choose a $G$-equivariant modulus 
	\[
		\mathfrak m  = \prod_{\mathfrak P } \mathfrak P^{n_{\mathfrak P}}
	\]
	such that:
	\begin{enumerate}
		\item 
		if $\mathfrak p$ is a prime that ramifies in $L/K$,
		then $\mathfrak P$ divides $\mathfrak m$;
		\item
		we have $n_{\mathfrak P} \equiv 1 \mod p^{k_{\mathfrak p}}$
		and $n_{\mathfrak P} \not= 1$ for every prime $\mathfrak P$ 
		above $p$;
		\item
		if $\zeta$ is a root of unity in $L$ such that 
		$\zeta \equiv 1 \mod \mathfrak m$, then $\zeta =1$.
	\end{enumerate}
	Then the $\Z_p[G]$-module $A_L^{\mathfrak m} \otimes_{\Z} \Z_p$ is
	cohomologically trivial.
\end{theorem}

\begin{proof}
	Using Theorem \ref{thm:units-ct}
	this may be proved along the lines of \cite[Theorem 1]{MR2805422}.
\end{proof}

\begin{remark}
	Condition (iii) on the modulus $\mathfrak m$ is satisfied
	when $\mathfrak m$ is divisble by at least two primes
	with different residue characteristic.
	In particular, whenever $L/K$ is weakly ramified above $p$,
	there exists a modulus $\mathfrak m$ with the above properties.
\end{remark}

\begin{remark}
	When $L/K$ is tamely ramified above $p$, a variant of this result
	has been established by the second author \cite[Theorem 1]{MR2805422}.
	This was an essential step in the proof of (the minus part of)
	the equivariant Tamagawa number conjecture for certain
	tamely ramified CM-extensions \cite{MR3552493}.
\end{remark}

\begin{remark}
	In characteristic $0$ it is easy to prove Theorem \ref{thm:units-ct}
	when we choose $n$ sufficiently large.
	However, in order to generalize the aforementioned results on
	the equivariant Tamagawa number conjecture to weakly ramified
	extensions, one most likely has to apply Theorem \ref{thm:ray-class-ct}
	for infinitely many Galois extensions $L_m/K_m$, $m \in \N$,
	where $L_m$ denotes the $m$-th layer in the cyclotomic $\Z_p$-extension
	of $L$. Moreover, one has to choose compatible moduli for each layer
	$m$ and this is only possible when we take the full strength of
	Theorem \ref{thm:ray-class-ct} (and thus of Theorem \ref{thm:units-ct})
	into account.
\end{remark}

\nocite*
\bibliography{formal-groups-bib}{}
\bibliographystyle{amsalpha}

\end{document}